\documentclass[12pt,a4paper]{amsart}

\usepackage[english]{babel}
\usepackage[T1]{fontenc}
\usepackage[utf8]{inputenc}

\usepackage{amsmath, amsthm, amscd, amsfonts}
\usepackage{amssymb, tikz}
\usepackage{mathtools}

\usepackage{hyperref}
\usepackage[capitalise]{cleveref}

\usepackage{setspace}
\newcommand{\lusim}[1]{\smash{\underset{\raisebox{1.2pt}[0cm][0cm]{$\sim$}}
{{#1}}}}

\newtheorem{theorem}{Theorem}[section]
\newtheorem{lemma}[theorem]{Lemma}
\newtheorem{proposition}[theorem]{Proposition}
\newtheorem{corollary}[theorem]{Corollary}

\theoremstyle{definition}
\newtheorem{definition}[theorem]{Definition}

\theoremstyle{remark}
\newtheorem{remark}[theorem]{Remark}

\newtheorem{question}[theorem]{Question}

\title[Completeness of the provability logic GL]{Completeness of the G\"odel-L\"ob provability logic for the  filter sequence of normal measures}

\author[M. Golshani and R. Zoghifard]{Mohammad  Golshani and Reihane Zoghifard}
\thanks{%
 The first author's research has been supported by a grant from
  IPM (No. 1400030417)  and
Iran National Science Foundation (INSF) (No. 98008254). He thanks Moti Gitik for his very helpful discussions and ideas for the result of Section \ref{proofmaintheorem}.}

\address{
School of Mathematics, Institute for Research in Fundamental Sciences (IPM), P.O. Box:
19395-5746, Tehran-Iran. }
\email{golshani.m@gmail.com}

\address{
School of Mathematics, Institute for Research in Fundamental Sciences (IPM), P.O. Box:
19395-5746, Tehran-Iran. }
\email{r.zoghi@gmail.com}

\subjclass[]{}


\begin{document}
\begin{abstract}
Assuming the existence of  suitable large cardinals, we show it is consistent that the Provability logic $\mathbf{GL}$ is complete with respect to the filter sequence of normal measures. This result answers a question of Andreas Blass
from 1990 and a related question of Beklemishev and  Joosten.
\end{abstract}
 \maketitle

\section{Introduction}

The G\"odel-L\"ob provability logic $\mathbf{GL}$ deals with the study of modality $\Box$ interpreted as the provability predicate in any formal theory $T$ that can describe the arithmetic of natural numbers, such as Peano arithmetic;
$\Box\varphi$ is read as $\varphi$ is provable in $T$.
It is proved by Segerberg \cite{segerberg} that $\mathbf{GL}$ is sound and complete with respect to the class of all transitive and conversely well-founded Kripke frames. In fact, it is adequate to consider frames that are finite transitive irreflexive trees.
Afterward, Esakia \cite{esakia} perceived that the modal operator $\lozenge$, interpreted as consistency in $T$, has the same behavior as the derivative operator in topological scattered spaces.
Then he proved that $\mathbf{GL}$ is (strongly) complete with respect to the class of all scattered spaces.

In 1990,  Blass \cite{blass} improved Esakia's result.  Instead of topological description, he interpreted modal operators over filters associated with specific uncountable cardinals, which is a  most natural viewpoint in set theory.
He showed the soundness of $\mathbf{GL}$ concerning some natural classes of filters.
Then he studied the completeness of $\mathbf{GL}$ for two classes of these filters: end-segment filters and closed unbounded (club) filters.
He proved that (in $ZFC$)  $\mathbf{GL}$ is complete concerning the end-segment filters.
His first completeness result implies the completeness of $\mathbf{GL}$ with respect to any ordinal $\alpha \geq \omega^\omega$ equipped with the interval (order) topology. This result was  independently proved by Abashidze \cite{abashidze87}.
Investigating the class of club filter, Blass proved
the completeness of $\mathbf{GL}$
by assuming the G\"{o}del's axiom of constructibility or, more precisely, Jensen's square principle $\Box_\kappa$ for all uncountable cardinals $\kappa < \aleph_\omega$. Building on some deep results of Harrington and Shelah \cite{shelah}, he also showed that the incompleteness of $\mathbf{GL}$ for club filters is equiconsistent with the existence of a Mahlo cardinal.

Abashidze-Blass theorem launches a new line of research for investigating the completeness of provability logic $\mathbf{GL}$ and also its polymodal extensions $\mathbf{GLP}$ with respect to the natural topologies on ordinals, e.g., see \cite{bagaria2019,beklemishev2013,beklemishev2014}.

In this paper, we answer a question of Blass \cite{blass}  by showing that
the provability logic $\mathbf{GL}$  consistently can be complete with respect to the filter sequence of normal measures.
For each ordinal  $\eta$ let
\begin{center}
$\mathcal{M}_\eta=\bigcap \{\mathcal{U}: \mathcal{U}$ is a normal measure on $\eta \}.$
\end{center}
Note that $\mathcal{M}_\eta$ is proper iff $\eta$ is a measurable cardinal, in which case $\mathcal{M}_\eta$
is a normal $\kappa$-complete filter on $\eta$. Also, it is easily seen that $X \subseteq \eta$
has positive measure with respect to $\mathcal{M}_\eta$ iff for at least one normal measure $\mathcal{U}$
on $\eta$ we have $X \in \mathcal{U}.$

We prove the following theorem.
\begin{theorem}
\label{maintheorem}
Assume there are infinitely many strong cardinals. Then there exists a generic extension of
the canonical core model in which the provability logic $\mathbf{GL}$ is complete with respect to
the filter sequence $\langle \mathcal{M}_\eta: \eta \in On \rangle$.
\end{theorem}
\begin{remark}
As it is shown in \cite{blass}, some large cardinals are needed to get the result; indeed the consistency of the statement
implies the existence of inner models for measurable cardinals $\kappa$ with $o(\kappa) \geq n,$ for all $n<\omega.$
\end{remark}
As a corollary, we obtain the following, which answers Question 16 from \cite{joosten}.
\begin{corollary}
\label{corollary}
Assuming the existence of infinitely many strong cardinals $\langle \kappa_n: n<\omega \rangle$, it is consistent that  $\mathbf{GL}$ is complete with respect to the
 ordinal space  $(\alpha, \tau_{M})$, where $\alpha\geq \sup_{n<\omega}\kappa_n$ and $\tau_M$ is the  topology corresponding to the filter sequence of normal measures $\vec{\mathcal{M}}_\kappa$ (see Section \ref{concludingremarks}).
\end{corollary}
The paper is organized as follows. In Section \ref{preliminaries} we collect some definitions and facts from provability logic and set theory, and then in Section \ref{proofmaintheorem} we complete the proof of
Theorem  \ref{maintheorem}. In the last Section \ref{concludingremarks}, we discuss the problem of strong completeness of  $\mathbf{GL}$ with respect to
the filter sequence of normal measures and conclude with some remarks.

\section{Some preliminaries}\label{preliminaries}


\subsection{Preliminaries from provability logic}

Let $\mathbb{P}$ be a set of propositional variables. The syntax of modal logic is obtained by adding the modal operator $\Box$ to propositional logic. So if $\varphi$ is a formula, then $\Box\varphi$ is a formula. As usual, $\lozenge$ is used as a shorthand for $\neg\Box\neg$ and $\bot$ for the logical constant ``false''.

The system $\mathbf{GL}$ is defined by the following axioms schemata and rules:
\begin{itemize}
	\item propositional tautologies,
	\item K. $\Box(\varphi\rightarrow\psi) \rightarrow (\Box\varphi\rightarrow\Box\psi)$,
	\item L\"ob. $\Box(\Box\varphi\rightarrow \varphi) \rightarrow\Box \varphi$,
	\item MP. $\vdash\varphi,\ \vdash \varphi\rightarrow\psi \Rightarrow \vdash \psi$,
	\item Nec. $\vdash\varphi \Rightarrow \Box\varphi$.
\end{itemize}

A Kripke frame is a pair $\mathfrak{F}=(W,R)$ where $W$ is a non-empty set
and
$R\subseteq W\times W$ is an accessibility relation. A Kripke model is a triple
$\mathfrak{M}=(W,R,\nu)$ where $\nu$ is a valuation function which assigns to each $p\in \mathbb{P}$ a subset of $W$. The valuation function $\nu$ is extended to all formulas as follows:
\begin{align*}
\nu(\neg\varphi)&= W - \nu(\varphi)\\
\nu(\varphi \wedge \psi) &= \nu(\varphi)\cap \nu(\psi)\\
\nu(\Box\varphi)&= \{w\in W\ |\ (\forall v\in W)\ wRv\rightarrow v\in \nu(\varphi)\}
\end{align*}
A formula $\varphi$ is valid in $\mathfrak{M}$ if $\nu(\varphi)=W$, also it is valid in $\mathfrak{F}$ if it is valid in every model based on $\mathfrak{F}$.

\begin{proposition}(Segerberg \cite{segerberg})
	\label{Kripke complete}
	$\mathbf{GL}$ is complete with respect to the class of all finite transitive irreflexive trees.
\end{proposition}

The Kripke completeness of $\mathbf{GL}$ can facilitate the method of proving the other completeness results.
To be more precise, for a given class of structures, instead of proving the completeness directly, one can find a way
to transform the validity from this class to the class of Kripke frames.
This idea is also used by Blass to give a sufficient condition for the completeness of $\mathbf{GL}$ with respect to any family of filters.
To this end, a particular class of trees named $\mathbf{K}_n$ is considered in \cite{blass} as a crossing point between these two classes.

For each fixed natural number $n$, the nodes of $\mathbf{K}_n$ consists of all finite sequences of pairs  $\langle (i_1,j_1), \dots,(i_k,j_k)\rangle$ where
$n>i_1>\dots> i_k\geq 0$ and $j_1,\dots,j_k\in\omega$ are arbitrary. The order of $\mathbf{K}_n$, denoted by $\lhd$, is the end extension order, thus $t$ extends $s$
iff $s \lhd t$.
So the root of  $\mathbf{K}_n$  is the empty sequence $\langle\rangle$, and the  height of the tree is $n$.
Also, each node with height $0<i \leq n$ has infinitely many immediate successors
of height $j$ for each $j<i$.

It is easy to see that each finite transitive tree $(W,R)$ with height $n$ is a bounded morphic image of $\mathbf{K}_n$. That is, there is an onto function $f$ from the nodes of  $\mathbf{K}_n$ to $W$ such that for any $s,t\in \mathbf{K}_n$ and $w\in W$ we have
\begin{itemize}
	\item $s \lhd t$ implies $f(s)R f(t)$,
	\item if $f(s)R w$, then there is $t\in \mathbf{K}_n$ such that $s \lhd t$ and $f(t)=w$.
\end{itemize}
It is easy to see that the validity of formulas is preserved under bounded morphic images.
So, by  Proposition \ref{Kripke complete} we have
\begin{itemize}
	\item if $\mathbf{GL}\vdash\varphi$, then $\varphi$ is valid in $\mathbf{K}_n$  for every $n$.
\end{itemize}

Suppose that
$\vec{\mathcal{F}}=\langle   \mathcal{F}_\alpha: \alpha \in On        \rangle$
is a family of filters where $\mathcal{F}_\alpha$ is a filter on $\alpha$, for each $\alpha\in On$.
A valuation $\nu$ on this family is a function which assigns a class of ordinals to each $p\in \mathbb{P}$.
Then the valuation function $\nu$ is extended to all formulas by the standard rules for Boolean connectives and the following for $\Box$ operator:
\begin{align*}
\nu(\Box\varphi)&= \{\alpha\ |\ \nu(\varphi)\in \mathcal{F}_\alpha\}.
\end{align*}
Then for the dual operator $\lozenge$ we have
$$\nu(\lozenge\varphi)=\{\alpha\ |\ \nu(\varphi) \text{ has positive measure w.r.t } \mathcal{F}_\alpha\}.$$

A formula $\varphi$ is $\vec{\mathcal{F}}$-valid if for every valuation $\nu$ on $\vec{\mathcal{F}}$ we have $\nu(\varphi)=On$.

In this paper, we are interested in the filter sequence of normal measures  $\vec{\mathcal{M}}=\langle \mathcal{M}_\alpha: \alpha \in On    \rangle,$ where for each $\alpha, $ $\mathcal{M}_\alpha$ is the intersection of all normal measures on $\alpha$.

Note that in $\vec{\mathcal{M}}$,
the formula $\lozenge \top$ determines the class of all measurable cardinals, reciprocally,
$\Box\bot$ defines the class of all non-measurable ordinals.
Furthermore, $\lozenge^n \top$ is true at an ordinal $\alpha$ if and only if $\alpha$ is a measurable cardinal with Mitchell order $\geq n$ (see Definition \ref{morder4}).

By showing that for any $\alpha$ and $A$ if
$A\in\mathcal{M}_\alpha$, then
$\{\beta<\alpha: A\cap\beta\in\mathcal{M}_\beta\}\in\mathcal{M}_\alpha$; Blass proved the following soundness theorem.
\begin{proposition}\cite[Theorem 2]{blass}
	$\mathbf{GL}$ is sound with respect to the class of normal filters $\vec{\mathcal{M}}$.
\end{proposition}

From the soundness result, one can indicate some properties of measurable cardinals. For example, Blass showed that the validity of the L\"ob formula implies that any measurable cardinal $\kappa$ has a normal measure containing
$\{\alpha<\kappa: \alpha \text{ is not measurable}\}$.
Also, if  $A$ has a positive measure with respect to $\mathcal{M}_\kappa$, then so does the set
$\{\alpha\in A : \alpha \text{ has no normal filter containing } A\}$.
More generally, one can see that the validity of
$\lozenge^{n+1}\top \rightarrow \lozenge(\lozenge^m\top\wedge\Box^{m+1}\bot)$ for each $m<n$, implies that any measurable cardinal $\kappa$ with $o(\kappa)\geq n$ has  a normal measure  containing
$\{\alpha<\kappa : o(\alpha)=m\}$, where $o(\alpha)$ is the Mitchell order of $\alpha$ (see Definition \ref{morder4}).

The following lemma gives a sufficient condition to  convert a Kripke interpretation of a given formula into a filter interpretation (see the proof of Theorem 3 in \cite{blass} for information on how this conversion is defined).
So, the main part of the proof of Theorem \ref{maintheorem} is to show that
 the following lemma holds for a family of normal filter sequences; the proof is given in Section \ref{proofmaintheorem}.
Then  the completeness of $\mathbf{GL}$ with respect to these filters is obtained by Proposition \ref{Kripke complete}.

\begin{lemma}
\label{reduction lemma to trees}
(Blass \cite{blass}) Let $\vec{\mathcal{F}}=\langle   \mathcal{F}_\alpha: \alpha \in On        \rangle$ be a family of
filters $\mathcal{F}_\alpha$ on $\alpha.$  Suppose that
for each $n<\omega$ there exists a function $\Gamma: \mathbf{K}_n \to \mathcal{P}(On)$ satisfying the following conditions:
\begin{enumerate}
\item $\Gamma(\langle \rangle)$ in non-empty,

\item if $s \neq t$ are in $\mathbf{K}_n $, then $\Gamma(s) \cap \Gamma(t)$ is empty,

\item If $s \lhd t$ are in  $\mathbf{K}_n $ and $\alpha \in \Gamma(s)$, then $\Gamma(t) \cap \alpha$ has positive measure with respect to $\mathcal{F}_\alpha$,

\item If $s \in \mathbf{K}_n $ and $\alpha \in \Gamma(s),$ then $\bigcup_{s \lhd t}\Gamma(t) \cap \alpha \in \mathcal{F}_\alpha$.
\end{enumerate}
Then every $\vec{\mathcal{F}}$-valid modal formula is provable in $\mathbf{GL}$.
\end{lemma}

\subsection{Preliminaries from set theory}
In this subsection we recall some definitions and facts about measurable cardinals and  their Mitchell order structure.
\begin{definition}
An uncountable cardinal
$\kappa$ is a measurable cardinal if there exists a $\kappa$-complete non-principal ultrafilter on $\kappa.$
 \end{definition}
 One can show that any measurable cardinal
$\kappa$  carries a \emph{normal measure}, i.e., a $\kappa$-complete non-principal ultrafilter $\mathcal{U}$ on $\kappa$ which is normal:
\[
\forall \xi < \kappa, A_\xi \in \mathcal{U} \Rightarrow \bigtriangleup_{\xi<\kappa}A_\xi=\{\alpha < \kappa: \forall\xi<\alpha,~\alpha \in A_\xi    \}\in \mathcal{U}.
\]
Given a normal measure $\mathcal{U}$ on $\kappa$ we can perform the ultrapower $\text{Ult}(V, \mathcal{U})$ and the ultrapower embedding $j: V \to \text{Ult}(V, \mathcal{U})$ which is defined by $j(x) = [c_x]_{\mathcal{U}}$, where $c_x: \kappa \to V$ is the constant function $x$. By \L{}o\'{s} theorem, $j$ is easily seen to be an elementary embedding. On the other hand, $\text{Ult}(V, \mathcal{U})$
is well-founded, and hence it is isomorphic to a unique transitive inner model $M_{\mathcal{U}}$ via a unique
 isomorphism
 $\pi: \text{Ult}(V, \mathcal{U}) \simeq M_{\mathcal{U}}$. Then $j_{\mathcal{U}}: V \to M_{\mathcal{U}}$, defined by $j_{\mathcal{U}}=\pi \circ j$ defines an elementary embedding from the universe $V$ into an inner model $M_{\mathcal{U}}$ with critical point $\kappa$ (i.e., $j_{\mathcal{U}} \restriction \kappa=\text{id}\restriction \kappa$ and $j_{\mathcal{U}}(\kappa)> \kappa$).

 Conversely, given a non-trivial elementary embedding $j: V \to M$ from $V$ into an inner model $M$ with critical
 point $\kappa,$ one can form the normal measure
 \[
 \mathcal{U}=\{ X \subseteq \kappa: \kappa \in j(X)           \}
 \]
on $\kappa$ and then $j$ factors through $j_{\mathcal{U}}$ in the sense that  $j=k \circ j_{\mathcal{U}}$, where $k: M_{\mathcal U} \to M$
is defined as $k([f]_{\mathcal U})=j(f)(\kappa)$.
\begin{definition}
\begin{enumerate}
\item[(a)] Suppose $\lambda \geq \kappa$ are uncountable cardinals. Then $\kappa$ is  $\lambda$-strong if there exists a non-trivial
elementary embedding $j: V \to M$ from $V$ into some inner model $M$ with critical point $\kappa$ such that $^{\kappa}M \subseteq M$,
$V_\lambda \subseteq M$ and $j(\kappa)> \lambda$.

\item[(b)] A cardinal $\kappa$ is strong if it is $\lambda$-strong for all $\lambda \geq \kappa.$
\end{enumerate}
\end{definition}
We now define an order on normal measures introduced by Mitchell.
\begin{definition}\label{Mitchel def}
(Mitchell \cite{mitchell}) Suppose $\kappa$ is a measurable cardinal and $\mathcal{U}, \mathcal{W}$ are normal measures on it. Then $\mathcal{W} \lhd \mathcal{U}$ if and only if $\mathcal{W} \in \text{Ult}(V, \mathcal{U}).$
\end{definition}
In \cite{mitchell}, Mitchell proved that $\lhd$ is a well-founded order now known as the Mitchell ordering. Thus given
any normal measure $\mathcal{U}$ on $\kappa,$ we can define its Mitchell order as
\[
o(\mathcal{U})=\sup\{o(\mathcal{W})+1: \mathcal{W} \lhd \mathcal{U}                \}.
\]
The Mitchell order of $\kappa$ is also defined as
\[
o(\kappa)=\sup\{o(\mathcal{U})+1: \mathcal{U} \text{~is a normal measure on~}\kappa    \}.
\]
\begin{definition}
\label{morder4}
Suppose $\kappa$ is a measurable cardinal. Then
\[
\lhd(\kappa)= (\{ \mathcal{U}: \mathcal{U} \text{~is a normal measure on~}\kappa   \}, \lhd).
\]
\end{definition}
The  structure of $\lhd(\kappa)$ is widely studied in set theory, we refer to
 \cite{omer1} and \cite{omer2}  for most recent results.
 We will need the following,
 which plays an essential role in the proof of Theorem \ref{maintheorem}.
\begin{theorem}
\label{structure of mitchell orders-main}
(Ben-Neria \cite{omer2}) Let $V=L[E]$ be a core model.
  Suppose  there is a strong cardinal $\kappa$ and infinitely many measurable cardinals above it. Let $(\mathbf{S}, <)$ be a countable well-founded
order of rank at most $\omega.$
Then there exists a generic extension $V^*$ of $V$ in which $\lhd(\kappa)^{V^*} \simeq (\mathbf{S}, <)$.
\end{theorem}
Let us recall that a forcing notion $(\mathbb{Q}, \leq, \leq^*)$ satisfies the Prikry property if $\leq^* \subseteq \leq$
and for each $p \in \mathbb{Q}$ and each sentence $\sigma$ in the forcing language of $(\mathbb{Q}, \leq)$, there exists $q \leq^* p$ such that $q$ decides $\sigma$.
The key point is that in general $(\mathbb{Q}, \leq)$ lacks to be even $\aleph_1$-distributive, but if $(\mathbb{Q}, \leq, \leq^*)$ satisfies the Prikry property and
$(\mathbb{Q}, \leq^*)$ is $\kappa$-closed, then forcing with $(\mathbb{Q}, \leq)$ does not add new bounded subsets to $\kappa.$
We may note that Ben-Neria's forcing of the above theorem   can be considered as a Prikry type forcing notion $(\mathbb{Q}, \leq, \leq^*)$, and furthermore, given any
 $\theta < \kappa$, we can manage the forcing so that $(\mathbb{Q}, \leq^*)$ is $\theta$-closed, in particular it does not add any new bounded subsets to $\theta$.

 The following is an immediate corollary of the above theorem, whose proof requires familiarity with
    Prikry type forcing notions and their iterations; see \cite{gitik} for more information.\footnote{The readers unfamiliar with forcing may skip the proof of this theorem.}

It is worth mentioning that the usual iteration of Prikry type forcing notions, say, for example, finite support iteration, fails to preserve cardinals. To give a motivation of why this happens, let $\langle \kappa_n: n<\omega \rangle$ be an increasing sequence of measurable cardinals and let $\kappa=\sup_{n<\omega}\kappa_n$. For each $n< \omega$, let $\mathcal U_n$ be a normal measure on $\kappa_n$ and let $\mathbb{P}_n$ be the corresponding Prikry forcing. Thus a condition in $\mathbb{P}_n$ is a pair $(s, A),$ where $s \in [\kappa_n]^{<\omega}, A \in \mathcal U_n$ and $\max(s) < \min(A)$. Given another condition $(t, B)$, we say $(t, B) \leq (s, A)$ iff $s \unlhd t, B \subseteq A$ and $t \setminus s \subseteq A.$ The direct extension order is defined as
 $(t, B) \leq^* (s, A)$ iff $(t, B) \leq (s, A)$ and $s=t$. Forcing with $\mathbb{P}_n$ adds an $\omega$-sequence $C_n \subseteq \kappa_n$,
  which is cofinal in $\kappa_n$, called the Prikry sequence. Now let $\mathbb{P}$ be the finite support product of $\mathbb{P}_n$'s.\footnote{The same argument applies to the case where $\mathbb{P}$ is a finite support iteration, but as it requires more details, we avoid it.} Thus a condition in
  $\mathbb{P}$  is a sequence $p= \langle  p_n: n<\omega       \rangle$, where each $p_n \in \mathbb{P}_n$, and for all but finitely many $n, p_n=(\emptyset, \kappa_n)=1_{\mathbb{P}_n}$.
 We claim that forcing with $\mathbb{P}$ collapses $\kappa$ into $\omega$. To see this, let $G$ be $\mathbb{P}$-generic over $V$, and for each $n$, let $C_n$ be the Prikry sequence added by it thorough $\kappa_n.$ Define $f: \omega \rightarrow \kappa$ by $f(n)=\min(C_n)$.
 Given any $\alpha < \kappa,$ it is easy to see that the set
 $$D_\alpha=\{p= \langle p_n: n<\omega\rangle \in \mathbb{P}: \exists n \big( \kappa_n > \alpha,~  p_n=(\{\alpha\}, \kappa_n \setminus (\alpha+1)) \big)   \}$$
  is dense in $\mathbb{P}$. But if $p \in D_\alpha$ is as above,  then $p \Vdash_{\mathbb{P}}$``$\lusim{f}(n)=\min(\lusim{C}_n)=\alpha$''
It follows that $\text{range}(f)=\kappa$, and thus $\kappa$ is collapsed.

However, there are other ways to iterate Prikry type forcing notions which avoid the above problem; the most known methods are the Magidor iteration (full support iteration) and Gitik iteration (Easton support iteration). Such kind of iterations
satisfy the Prikry property, and by using this fact, one can show that the resulting forcing behaves nicely.

\begin{theorem}
\label{structure of mitchell orders}
 Let $V=L[E]$ be a core model \footnote{In the proof of  Theorem \ref{structure of mitchell orders-main}, it is required  the ground model to be a canonical core model
of the form $V=L[E]$. This is needed,  as structural properties of the core model
are used for several arguments. Since
our proof of Theorem \ref{structure of mitchell orders} is based on Theorem \ref{structure of mitchell orders-main}, so we also need to start with a canonical core model.}. Suppose  there is an $\omega$-sequence $\langle \kappa_n: n<\omega \rangle$ of strong cardinals and suppose
$\langle (\mathbf{S}_n, <_n): n<\omega \rangle$ is a  sequence of countable well-founded
orders, each of rank at most $\omega.$
Then there exists a generic extension $V^*$ of $V$ in which for each $n<\omega$, $\lhd(\kappa_n)^{V^*} \simeq (\mathbf{S}_n, <_n)$.
\end{theorem}

\begin{proof}
Recall from the above  that Ben-Neria's forcing of Theorem \ref{structure of mitchell orders-main}  is a Prikry type forcing notion $(\mathbb{Q}, \leq, \leq^*)$.
Now let
 $\langle \kappa_n: n<\omega \rangle$ and
$\langle (\mathbf{S}_n, <_n): n<\omega \rangle$ be as above. For each $n<\omega$ let $\langle \lambda^n_i: i \leq \omega \rangle$ be the first
$\omega+1$ measurable cardinals above $\kappa_n$.\footnote{The need for the use of measurable cardinals $\lambda^n_i$, for $i<\omega$, comes from Ben-Neria's theorem \ref{structure of mitchell orders-main}, we choose $\lambda^n_\omega$ large enough above $\lambda^n_i$'s, $i<\omega$, to get enough closure property (with respect to the direct extension order).} Then  for each $n$, $\lambda^n_{\omega} < \kappa_{n+1}$. Let
\[
\mathbb{P}=\langle   \langle (\mathbb{P}_n, \leq_{\mathbb{P}_n}, \leq^*_{\mathbb{P}_n}): n \leq \omega \rangle, \langle (\lusim{\mathbb{Q}_n}, \leq_{\lusim{\mathbb{Q}_n}}, \leq^*_{\lusim{\mathbb{Q}_n}}): n<\omega        \rangle\rangle
\]
be the Magidor iteration of Prikry type forcing notions, where for each $n<\omega$, $(\mathbb{Q}_n, \leq_{\lusim{\mathbb{Q}_n}}, \leq^*_{\lusim{\mathbb{Q}_n}})$ is defined in $V^{\mathbb{P}_n}$ such that:
\begin{enumerate}
\item $|\mathbb{Q}_n| < \kappa_{n+1}$,

\item $(\mathbb{Q}_n, \leq^*_{\mathbb{Q}_n})$ is $\lambda^{n-1}_\omega$-closed, in particular it adds no new bounded subsets to $\lambda^{n-1}_\omega$,

\item  $\lhd(\kappa_n)^{V^{\mathbb{P}_{n+1}}} \simeq (\mathbf{S}_n, <_n)$,
\end{enumerate}
This is possible by  Theorem \ref{structure of mitchell orders-main} (and its proof) and the fact that by (1), all cardinals $\kappa_m,$ for $m>n$ remain strong in the extension by $\mathbb{P}_{n+1}$.
Then $V^*=V^{\mathbb{P}_\omega}$ is the required model.
\end{proof}

\section{Completeness of $\mathbf{GL}$ with respect to the normal filter sequence}
\label{proofmaintheorem}
In this section we prove Theorem \ref{maintheorem}.
Let $L[E]$ be the canonical extender model and suppose in it there is an $\omega$ sequence $\langle  \kappa_n: 0<n<\omega     \rangle$ of strong cardinals. By  Theorem \ref{structure of mitchell orders}, we can extend $L[E]$ to a generic extension $V$ in which
 the structure of the Mitchell order of $\kappa_n$, $\lhd(\kappa_n)$, is isomorphic to $\mathbf{S}_n,$ where
$\mathbf{S}_n= \mathbf{K}_n \setminus \{ \langle \rangle  \}$, ordered by $t < s$ iff  $t$ end extends $s$.

We show that in $V$, the provability logic $\mathbf{GL}$ is complete with respect to the normal filter sequence.  Set  $\kappa=\sup_{n<\omega}\kappa_n$.
By Lemma \ref{reduction lemma to trees}, it suffices to  show that for each $n<\omega$ there exists a function $\Gamma: \mathbf{K}_n \to \mathcal{P}(\kappa)$ satisfying the following conditions:
\begin{enumerate}
\item[$(\dagger)_1$] $\Gamma(\langle \rangle)$ is non-empty,

\item[$(\dagger)_2$] if $s \neq t$ are in $\mathbf{K}_n$, then $\Gamma(s) \cap \Gamma(t)$ is empty,

\item[$(\dagger)_3$] If $s \lhd t$ are in  $\mathbf{K}_n$ and $\eta \in \Gamma(s)$, then $\Gamma(t) \cap \eta$ has positive measure with respect to $\mathcal{M}_\eta$,
i.e., $\Gamma(t) \cap \eta$ belongs to at least one normal measure on $\eta$,

\item[$(\dagger)_4$] If $s \in \mathbf{K}_n$ is not maximal and $\eta \in \Gamma(s),$ then $\bigcup_{s \lhd t}\Gamma(t) \cap \eta \in \mathcal{M}_\eta$.
\end{enumerate}

Let us first  suppose that $n=1$. Let $\mathbf{S}=\mathbf{S}_1$ and $\eta=\kappa_1$.
Then $\mathbf{S}=\{\langle(0, \ell)\rangle: \ell<\omega     \}$, and in $V$,  $\eta$
has exactly $\omega$-many normal  measures $\mathcal{U}(0, \ell), \ell<\omega,$ all of Mitchell order $0$.
Pick  sets $A_{0, \ell} \in \mathcal{U}(0, \ell),$ so that for all $\ell \neq \ell', A_{0, \ell} \cap A_{0, \ell'}=\emptyset.$

Define $\Gamma: \mathbf{K}_1 \to \mathcal{P}(\kappa)$ by
\begin{center}
 $\Gamma(s)=$ $\left\{
\begin{array}{l}
         \{ \eta \}  \hspace{1.6cm} \text{ if } s= \langle \rangle,\\
         A_{0, \ell}  \hspace{1.5cm} \text{ if } s=\langle (0,\ell)  \rangle.
     \end{array} \right.$
\end{center}
It is clear that $\Gamma$ is as required.

Now suppose that $n \geq 2$. Let  $\mathbf{S}=\mathbf{S}_n$ and $\eta=\kappa_n$.
Thus in $V$, $\lhd(\eta)\simeq \mathbf{S}.$ Let
\[
\lhd(\eta) = \{\mathcal{U}(s): s \in \mathbf{S}   \},
\]
where for each $s, t \in \mathbf{S}$
\[
t < s \iff \mathcal{U}(t) \lhd \mathcal{U}(s).\footnote{Recall that $t < s$ iff $s \lhd t$.}
\]
For each $s \in \mathbf{S}$ let $j_s: V \to M_s \simeq \text{Ult}(V, \mathcal{U}(s))$ be the canonical ultrapower embedding. Note that for each
$X \subseteq \eta,$ we have
\[
X \in \mathcal{U}(s) \iff \eta \in j_s(X).
\]
Pick the  sets $A_{s}$ for $s \in \mathbf{S}$
such that:
\begin{enumerate}
\item[$(\beth)_1$] $A_{s} \in \mathcal{U}(s),$

\item[$(\beth)_2$] for all $s \neq t$ in $\mathbf{S}, A_s \cap A_t=\emptyset.$
\end{enumerate}
For $t < s$ in $\mathbf{S}$, let $g_t^s: \eta \to V$ represent $\mathcal{U}(t)$
in the ultrapower by $\mathcal{U}(s)$, i.e.,
$\mathcal{U}(t)=[g_t^s]_{\mathcal{U}(s)}$.
The next lemma is proved in \cite{mitchell}.
\begin{lemma}
\label{representationn1}
Suppose $t< s$ are in $\mathbf{S}$ and $X \subseteq \eta.$ Then
\[
X \in \mathcal{U}(t) \iff \{\nu \in A_s: X \cap \nu \in g_t^s(\nu)              \} \in \mathcal{U}(s).
\]
\end{lemma}
\begin{proof}
We give a proof for completeness. Let $X \subseteq \eta$ and set $Y= \{\nu \in A_s: X \cap \nu \in g_t^s(\nu)              \}.$ Then
$j_s(X) \cap \eta=X$ and $j_s(g_t^s)(\eta)=\mathcal{U}(t)$,  hence
\[
Y \in \mathcal{U}(s) \iff \eta \in j_s(Y) \iff j_s(X) \cap \eta \in j_s(g_t^s)(\eta) \iff X \in \mathcal{U}(t),
\]
which gives the result.
\end{proof}
The proof of the next lemma follows the ideas of \cite{magidor}.
\begin{lemma}
\label{representationn2}
Suppose $u< t< s$ are in $\mathbf{S}$. Then
\[
A^1_{s, t, u}=\{  \nu \in A_s: g_u^s(\nu) \lhd g_t^s(\nu) \text{~are normal measures on~}\nu\} \in \mathcal{U}(s).
\]
\end{lemma}
\begin{proof}
As $j_s(g_u^s)(\eta)=\mathcal{U}(s)$ and  $j_s(g_t^s)(\eta)=\mathcal{U}(t)$, we have
\[
A^1_{s, t, u} \in \mathcal{U}(s) \iff \eta \in j_s(Y) \iff M_s \models \mathcal{U}(s) \lhd \mathcal{U}(t) \text{~are normal measures on~} \eta,
\]
which gives the required result.
\end{proof}
Suppose $u< t< s$ are in $\mathbf{S}$ and $g_u^s(\nu) \lhd g_t^s(\nu).$ Then $g_u^s(\nu)$ has a representative function which presents it
 in the ultrapower by $g_t^s(\nu)$. The next lemma shows that there is already a canonical such representation.
\begin{lemma}
\label{representationn3}
Suppose $u< t< s$ are in $\mathbf{S}$. Then
\[
A^2_{s, t, u}=\{  \nu \in A_s: g_u^s(\nu) =[g_u^t \restriction \nu]_{g_t^s(\nu)} \} \in \mathcal{U}(s).
\]
\end{lemma}
\begin{proof}
As in the proof of Lemma \ref{representationn2}, we have
\[
A^2_{s, t, u} \in \mathcal{U}(s) \iff M_s \models \mathcal{U}(u)= [j_s(g_u^t)\restriction \eta]_{\mathcal{U}(t)}.
\]
On the other hand $j_s(g_u^t)\restriction \eta=g_u^t$ and hence
 $[j_s(g_u^t)\restriction \eta]_{\mathcal{U}(t)}=[g_u^t]_{\mathcal{U}(t)}$, from which the result follows.
\end{proof}
For $u<t<s$ in $\mathbf{S}$, let $A_{s,t, u}=A^1_{s, t, u} \cap A^2_{s, t, u}$.
The next lemma is an immediate corollary of the above two lemmas.
\begin{lemma}
\label{representationn4}
Suppose $s \in \mathbf{S}$. Then
\[
B_s= \bigcap_{u<t<s}A_{s, t, u} \in \mathcal{U}(s).
\]
\end{lemma}

For each $s \in \mathbf{S},$
set
\[
\mathbf{S}/(<s) = \{ t \in \mathbf{S}: t< s            \}.
\]
\begin{lemma}
\label{representationn5}
\begin{enumerate}
\item[(a)] Suppose $s \in \mathbf{S}$ is a minimal node. Then
\[
C_{s}=\{\nu \in B_s: \nu \text{~is an inaccessible non-measurable cardinal~}   \} \in \mathcal{U}(s).
\]

\item[(b)] Suppose $s \in \mathbf{S}$ is not minimal. Then
\[
C_s=\{\nu \in B_s: \lhd(\nu) \simeq \mathbf{S}/(<s)   \} \in \mathcal{U}(s).
\]
furthermore, for each $\nu \in C_s,$
\[
\lhd(\nu) = \{ g_t^s(\nu): t < s                        \}.
\]
\end{enumerate}
\end{lemma}
\begin{proof}
(a) Clearly, $\{\nu \in B_s: \nu \text{~is an inaccessible cardinal~}   \} \in \mathcal{U}(s).$
Now suppose by contradiction,
$Y=\{\nu \in B_s: \nu \text{~is a measurable cardinal~}   \} \in \mathcal{U}(s).$
For each $\nu \in Y$ pick some normal measure $\mathcal{W}_\nu$  on $\nu$ and set $\mathcal{W}=[\mathcal{W}_\nu: \nu \in Y]_{\mathcal{U}(s)}.$
Then $\mathcal{W}$ is a normal measure on $\nu$ and $\mathcal{W} \lhd \mathcal{U}(s)$. This contradicts our choice of the Mitchell order structure of
$\lhd(\eta)$.

(b) We show that
\[
\{  \nu \in B_s: \lhd(\nu) = \{ g_t^s(\nu): t < s                        \}              \} \in \mathcal{U}(s).
\]
Suppose not. Then there exists a measure one set $Y \in \mathcal{U}(s)$ such that for each
$\nu \in Y,$ there exists a  normal measure $\mathcal{W}_\nu$ on $\nu$ such that $\mathcal{W}_\nu \notin \{ g_t^s(\nu): t < s                        \}.$ Set $\mathcal{W}=[\mathcal{W}_\nu: \nu \in Y]_{\mathcal{U}(s)}.$
Then $\mathcal{W} \lhd \mathcal{U}_s$ is a normal measure on $\nu$ and $\mathcal{W} \neq \mathcal{U}(t)$  for all $t<s$. This contradicts our choice of the Mitchell order structure of
$\lhd(\eta)$ below $\mathcal{U}(s)$.
\end{proof}
Using Lemma \ref{representationn1}, and by shrinking the sets $C_s, s \in \mathbf{S}$, we may assume that:
\begin{enumerate}
\item[$(\beth)_3$] for all $t<s$ in $\mathbf{S}$ and all $\nu \in C_s,$
$C_t \cap \nu \in g_t^s(\nu).$
\end{enumerate}

Define $\Gamma: \mathbf{K}_n \to \mathcal{P}(\kappa)$ by
\begin{center}
 $\Gamma(s)=$ $\left\{
\begin{array}{l}
         \{\eta\}  \hspace{1.6cm} \text{ if } s=\langle \rangle,\\
         C_s  \hspace{1.75cm} \text{ if }s \neq \langle \rangle,\\
     \end{array} \right.$
\end{center}
\begin{lemma}
\label{gammaasrequired}
$\Gamma$ satisfies the requirements $(\dagger)_1$-$(\dagger)_4.$
\end{lemma}
\begin{proof}
Clearly clause $(\dagger)_1$ is satisfied as $\eta \in \Gamma(\langle \rangle)$ and clause $(\dagger)_2$
follows from $(\beth)_2$ and the fact that $C_s \subseteq A_s,$ for each $s \in \mathbf{S}$.

To show that clause $(\dagger)_3$ is satisfied, let $s  \lhd t$ be in $\mathbf{K}_n$ and $\nu \in \Gamma(s)$. If $s=\langle \rangle,$ then
$\nu=\eta$ and we have $\Gamma(t) =C_t \in \mathcal{U}(t)$, in particular,  $\Gamma(t)$ has positive measure with respect to $\mathcal{M}_\eta.$
 If $s \neq \langle \rangle,$ then by $(\beth)_3$, $\Gamma(t) \cap \nu=C_t \cap \nu \in g_t^s(\nu)$, and by Lemma \ref{representationn2}, $g_t^s(\nu)$ is a normal measure
on $\nu$. Thus $\Gamma(t) \cap \nu$ has positive measure with respect to $\mathcal{M}_\nu$, as required.

Finally to see that  clause $(\dagger)_4$ is satisfied, let $s \in \mathbf{K}_n$ be a non-maximal element and let $\nu \in \Gamma(s)$.  First suppose that $s=\langle \rangle.$ Then $\nu=\eta$, and
 \begin{itemize}
 \item[$(\eta)_1$] the only normal measures on $\eta$ are $\mathcal{U}(t), t \in \mathbf{S}$,

 \item[$(\eta)_2$] for all $t \in \mathbf{S}, C_t \in \mathcal{U}(t)$.
 \end{itemize}
 It immediately follow that
 \[
 \bigcup_{\langle \rangle \lhd t} \Gamma(t)=  \bigcup_{\langle \rangle \lhd t} C_t \in \bigcap_{t \in \mathbf{S}} \mathcal{U}(t)=\mathcal{M}_\eta.
 \]
Now suppose that $s \neq \langle \rangle.$ Then
\begin{itemize}
 \item[$(\nu)_1$] by Lemma \ref{representationn5}(b), the only normal measures on $\nu$ are $g_t^s(\nu)$ where $t \rhd s$,

 \item[$(\nu)_2$] by $(\beth)_3$, for all $s \lhd t, C_t \cap \nu \in g_t^s(\nu)$.
 \end{itemize}
Thus
 \[
 \bigcup_{s \lhd t} \Gamma(t) \cap \nu=  \bigcup_{s \lhd t} C_t \cap \nu \in \bigcap_{s \lhd t} g_t^s(\nu)=\mathcal{M}_\nu.
 \]
\end{proof}
Theorem \ref{maintheorem} follows.

\section{Concluding remarks}
\label{concludingremarks}
Although $\mathbf{GL}$ is not strongly complete with respect to Kripke semantics, interpreting $\lozenge$ as the derivative operator makes $\mathbf{GL}$ strongly complete over scattered spaces, specifically with respect to any ordinal $\alpha > \omega^\omega$ equipped  with the interval topology.
However, $\mathbf{GL}$ is not strongly complete concerning filter sequence of normal measures.
To see this consider the set
$\Sigma=\{\lozenge p_0\}\cup \{\Box(p_i\rightarrow\lozenge p_{i+1})\ |\ i<\omega\}$
and suppose that there is a valuation $\nu$ on $\vec{\mathcal{M}}$ such that
$\kappa$ satisfies $\Sigma$. Thus,
$\nu(p_i\rightarrow \lozenge p_{i+1})\in \mathcal{M}_\kappa$, for each $i<\omega$.
The truth of $\lozenge p_0$ and
$\Box p_0\rightarrow \lozenge p_1$
in $\kappa$ implies that there is a normal measure $\mathcal{U}_0$ on $\kappa$ such that
$\nu(p_0),\nu(\lozenge p_1)\in \mathcal{U}_0$.
Let $o(\mathcal{U}_0)=\alpha_0$. Then there exists a normal measure $\mathcal{U}_1$ such that $\nu(p_1), \nu(\lozenge p_2)\in \mathcal{U}_1$ and $o(\mathcal{U}_1)<o(\mathcal{U}_0)$.
By induction, we can see that for each $i$, there is a normal measure $\mathcal{U}_i$ such that $\nu(p_i), \nu(\lozenge p_{i+1})\in \mathcal{U}_i$ and
$o(\mathcal{U}_i)=\alpha_i<\alpha_{i-1}$. This gives a strictly decreasing sequence $\langle \alpha_i: i<\omega \rangle$
of ordinals, which is a contradiction.
Furthermore,  in \cite{aguilera2016} (Corollary 2.7), it is generally shown that $\mathbf{GL}$ is not strongly complete with respect to topologies on ordinals based on countably complete filters, such as club filters and measurable filters.

Note that we can consider a filter sequence of normal measures $\vec{\mathcal{M}_{\kappa}}$, the restriction of $\vec{\mathcal{M}}$ to any cardinal $\kappa$, as a topological space with a unique topology $\tau_{M}$ generated by the following sets:
\begin{itemize}
	\item if $\alpha<\kappa$ is not a measurable cardinal, then $\alpha$ is an isolated point,
	\item if $U\subseteq [0,\kappa]$, then $U\in \tau$ iff for any measurable cardinal $\alpha \in U$ there is $X\in \mathcal{M}_\alpha$ such that $X\subseteq U$.
\end{itemize}

For any $A\subseteq [0,\kappa]$, the set of limit points of $A$, denoted by $d(A)$, is the set of all ordinals $\alpha$ such that $A\cap\alpha$ has positive measure with respect to $\mathcal{M}_\alpha$.
Let $o(\alpha)=0$ if $\alpha$ is not measurable, then for any ordinal $\alpha\leq\kappa$ we have $\rho(\alpha)=o(\alpha)$, where $\rho$ is the derivative topological rank of the space $\vec{\mathcal{M}_{\kappa}}$, i.e, the least ordinal $\xi$ such that $\alpha\notin d^{\xi+1}(\vec{\mathcal{M}_{\kappa}})$.
Therefore, Corollary \ref{corollary} is obtained from Theorem \ref{maintheorem} for the space $(\kappa,\tau_{M})$ for sufficient large cardinal $\kappa$.

In \cite{aguilera2015} it is proved that for any given scattered space $\mathfrak{X}=(X,\tau)$ of sufficiently large derivative rank, $\mathbf{GL}$ is strongly complete with respect to $\mathfrak{X}_{+\lambda}=(X,\tau_{+\lambda})$ where  $\tau_{+\lambda}$ is a finer topology named Icard topology.
In particular, for filter sequence of normal measures and for $\lambda=1$, it is consistent that $\mathbf{GL}$ is strongly complete with respect to $(\kappa,\tau_{M+1})$ whenever
$\kappa$ is a measurable cardinal with $o(\kappa)\geq \omega^\omega+1$ and
$\tau_{M+1}$ is the generalized Icard topology, i.e., the least topology extending $\tau_{M}$ by adding all sets of the form
$\{\alpha<\kappa: \zeta<\rho(\alpha)\leq\xi\}$
for all $-1\leq\zeta<\xi\leq o(\kappa)$.
However, the strong completeness of \cite{aguilera2015} is based on the assumption that the set of propositional variables $\mathbb{P}$ is countable, and it remains open to find a natural topological space $\mathfrak{X}$ with respect to which $\mathbf{GL}$ is strongly complete based on uncountable language.

As it is shown in \cite{blass}, the incompleteness of $\mathbf{GL}$ with respect to club filters is equiconsistent with the existence of a Mahlo cardinal. However, the following question remains open.
\begin{question}
What is the exact consistency strength of `` $\mathbf{GL}$ is complete with respect to the filter sequence of normal measures ''$?$
\end{question}

As it is stated by Blass, for $\mathbf{GL}$ to be complete with respect to normal filters, we need the existence of measurable cardinals of all finite Mitchell orders, so the existence of some large cardinals is needed.
Our proof is based on a result from Ben-Neria \cite{omer2}, Theorem \ref{structure of mitchell orders-main}, although it seems that the assumption he has used is more than what we need, for now, this is the best possible result.
On the other hand, in order to make the article comprehensible to readers who are not familiar with the advanced concepts of set theory, we used a different and slightly stronger assumption than Ben-Neria's.
Note that we do not need our cardinals to be strong, but it suffices to be $\lambda$-strong for a suitable $\lambda$.
Also, as it is stated by Donder (see \cite{blass}), the existence of  large cardinals alone is not sufficient for our proof; for example, $\mathbf{GL}$ is incomplete  in the known canonical core models for strong cardinals, as in such models, for measurable cardinals $\kappa$ of Mitchell order $1$, $\mathcal{M}_\kappa$ is an ultrafilter, and this prevents $\mathbf{GL}$ from being complete.


\begin{thebibliography}{100}
	
	\bibitem{abashidze87}
	Merab Abashidze.
	\newblock Ordinal completeness of the G\"odel-L\"ob modal system.
	\newblock in: {\em Intensional Logics and the Logical Structure of Theories,}
	Metsniereba, Tbilisi, (1985), 49-73 (in Russian).
	
	\bibitem{aguilera2016}
	Juan~P Aguilera.
	\newblock A Topological Completeness Theorem for Transfinite Provability Logic.
	\newblock {\em arXiv preprint arXiv:1609.03074.},
	 (2016).
	
	\bibitem{aguilera2015}
	Juan~P Aguilera and David Fern{\'a}ndez-Duque.
	\newblock Strong completeness of provability logic for ordinal spaces.
	\newblock {\em  J. Symbolic Logic},
	 82 (2017), no. 2, 608–628.

 	

	
	\bibitem{bagaria2019}
	Joan Bagaria.
	\newblock Derived topologies on ordinals and stationary reflection.
	\newblock {\em  Trans. Amer. Math. Soc.},
	371 (2019), no. 3, 1981-2002.
	
    \bibitem{joosten}
	 Lev Beklemishev and Joost J. Joosten .
	\newblock Problems collected at the Wormshop 2012 in Barcelona.
	\newblock $http://www.mi$-$ras.ru/~bekl/Problems/worm_{-}problems.pdf$.

	\bibitem{beklemishev2009ordinal}
	Lev Beklemishev.
	\newblock Ordinal completeness of bimodal provability logic $GLB$.
	\newblock in: {\em International Tbilisi Symposium on Logic, Language, and
		Computation}, Springer, (2009), 1-15.
	
	\bibitem{beklemishev2013}
	Lev Beklemishev and David Gabelaia.
	\newblock Topological completeness of the provability logic $GLP$.
	\newblock {\em Ann. Pure Appl. Logic},
 164 (2013), no. 12, 1201-1223.
	
	\bibitem{beklemishev2014}
	Lev Beklemishev and David Gabelaia.
	\newblock Topological interpretations of provability logic.
	\newblock in: {\em Leo Esakia on duality in modal and intuitionistic logics},
	 Springer, (2014), 257-290.

	
	\bibitem{omer2}
	Omer Ben-Neria.
	\newblock The structure of the mitchell order-II.
	\newblock {\em  Ann. Pure Appl. Logic}, 166 (2015), no. 12, 1407-1432.
	
	\bibitem{omer1}
	Omer Ben-Neria.
	\newblock The structure of the mitchell order-I.
	\newblock {\em Israel J. Math.} 214 (2016), no. 2, 945–982.
	
	\bibitem{blass}
	Andreas Blass.
	\newblock Infinitary combinatorics and modal logic.
	\newblock {\em  J. Symbolic Logic}, 55 (1990), no. 2, 761-778.
	
	\bibitem{esakia}
	Leo Esakia.
	\newblock Diagonal constructions, L{\"o}b’s formula and Cantor’s scattered
	spaces.
	\newblock {\em Studies in logic and semantics},
	 Metsniereba, Tbilisi,  132 (1981),  128–143 (In Russian).

     \bibitem{fernandez14}
     David Fern{\'a}ndez-Duque.
     \newblock The polytopologies of transfinite provability logic.
     \newblock {\em Archive for Mathematical Logic 53}, 3-4 (2014), 385--431.

\bibitem{gitik}
Moti Gitik.
\newblock Prikry-type forcings.
\newblock {\em Handbook of set theory. Vols. 1, 2, 3, 1351–1447, Springer, Dordrecht, 2010.}

\bibitem{shelah}
	Leo Harrington and Saharon Shelah.
	\newblock Some exact equiconsistency results in set theory.
	\newblock {\em Notre Dame J. Formal Logic} 26 (1985), no. 2, 178–188.
	
	\bibitem{magidor}
	Menachem Magidor.
	\newblock Changing cofinality of cardinals.
	\newblock {\em Fund. Math.} 99 (1978), no. 1, 61–71.
	
	\bibitem{mitchell}
	William~J Mitchell.
	\newblock Sets constructible from sequences of ultrafilters.
	\newblock {\em  J. Symbolic Logic}, 39 (1974), 57–66.
	
	\bibitem{segerberg}
	Krister Segerberg.
	\newblock An essay in classical modal logic.
	\newblock {\em Filosofiska F\"oreningen och	Filosofiska Institutionen vid Uppsala Universitet}, 1971.
	
\end{thebibliography}
\end{document}